\documentclass[11pt]{amsart}
\usepackage{amsmath} 
\usepackage{amsthm}
\usepackage{amssymb} 
\usepackage{amsfonts}
\usepackage{graphicx} 
\usepackage{inputenc}
\usepackage{upref} 
\usepackage{mathptmx}
\usepackage[T1]{fontenc}
\usepackage{bm}
\usepackage{calc}

\makeatletter
\newtheorem*{rep@theorem}{\rep@title}
\newcommand{\newreptheorem}[2]{%
\newenvironment{rep#1}[1]{%
 \def\rep@title{#2 \ref{##1}}%
 \begin{rep@theorem}}%
 {\end{rep@theorem}}}
\makeatother

\newtheorem{theorem}{Theorem}[section]
\newtheorem{lemma}[theorem]{Lemma}
\newtheorem{definition}[theorem]{Definition}
\newtheorem{definitions}[theorem]{Definitions}
\newtheorem{proposition}[theorem]{Proposition}
\newreptheorem{theorem}{Theorem}
\newreptheorem{proposition}{Proposition}
\newreptheorem{lemma}{Lemma}
\newtheorem{corollary}[theorem]{Corollary}
\newtheorem{conjt}[theorem]{Conjecture}
\newtheorem{proj}[theorem]{Project}
\newtheorem{projs}[theorem]{Projects}
\newtheorem{remark}[theorem]{Remark}
\newtheorem{remarks}[theorem]{Remarks}

\title[Pentagonal Tilings]{Perimeter-minimizing Tilings by Convex and Non-convex Pentagons} 
\author[W. Ghang]{Whan Ghang}
\author[Z. Martin]{Zane Martin}
\author[S. Waruhiu]{Steven Waruhiu}

\begin{document}
\maketitle

\begin{abstract}
We study the presumably unnecessary convexity hypothesis in the theorem of Chung et al. \cite{pen11} on perimeter-minimizing planar tilings by convex pentagons. We prove that the theorem holds without the convexity hypothesis in certain special cases, and we offer direction for further research. 
\end{abstract}


\section{Introduction}
\label{intro} 
 
\subsection{Tilings of the plane by pentagons}
Chung et al. \cite[Thm. 3.5]{pen11} proved that certain ``Cairo'' and ``Prismatic'' pentagons provide least-perimeter tilings by (mixtures of) convex pentagons, and they conjecture that the restriction to convex pentagons is unnecessary. In this paper we consider tilings by mixtures of convex and non-convex pentagons, and we prove that under certain conditions the convexity hypothesis in the results of Chung et al. can in fact be removed. The conjecture remains open. 

\begin{figure}
	\centering
	\includegraphics[scale=.4]{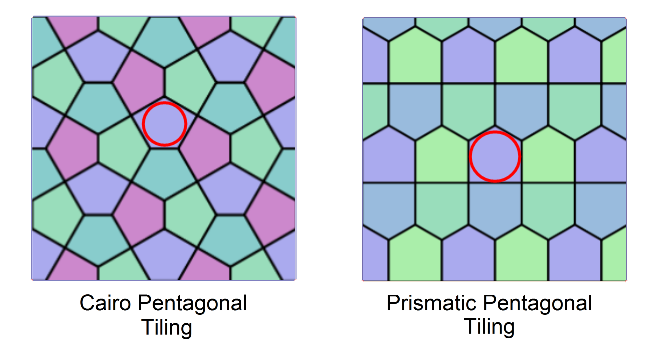}
		\caption{Tilings by Cairo and Prismatic pentagons provide least-perimeter tilings by unit-area convex pentagons. Can the convexity hypothesis be removed?}
	\label{fig:cairoprismatic}
\end{figure}

Throughout the paper, we assume all tilings are unit area and edge-to-edge. We focus on tilings of flat tori, although Section 5 begins the extension of our results to the plane by limit arguments. Our main results are Theorems \ref{oneeffn&mTorus} and \ref{dihedraltype2n&mTorus}. The first shows that tilings by an efficient pentagon and non-convex quadrilaterals cannot have less perimeter than Cairo or Prismatic tilings. The second shows that dihedral tilings by efficient pentagons and Type 2 non-convex pentagons cannot have less perimeter than Cairo or Prismatic tilings. 

The general strategy employed in our main results begins with the assumption that a mixed tiling with convex pentagons and non-convex pentagons (or in Section 3, quadrilaterals) exists that has less perimeter than a Cairo or Prismatic tiling. The first step in the proof is to show that such tilings must have at least one degree-four efficient vertex (Props. \ref{gendihedralSomeDeg4Verts} and \ref{type2dihedralSomeDeg4Verts}).

We then obtain a large lower bound on the ratio of  efficient to non-convex pentagons (or quadrilaterals in Section 3; Props. \ref{generalQuadRatio} and \ref{type2-ratio}). This is primarily done by showing that in order to tile the plane, the convex pentagons must have perimeter substantially higher than the regular pentagon, though a high bound on the perimeter of the non-convex pentagons would also suffice. Second we show (Thms. \ref{oneeffn&mTorus} and \ref{dihedraltype2n&mTorus}) that the ratio of convex pentagons to non-convex pentagons (or quadrilaterals) has an upper bound by bounding the number of efficient vertices and counting the number of angles appearing at such vertices. We derive a contradiction by showing that the upper bound is less than the lower bound, and thus conclude that a tiling cannot have perimeter less than that of a Cairo/Prismatic tiling.

In addition to our main results, we categorize non-convex pentagons in Proposition \ref{3typesnonconvex}, and we bound the angles and edge-lengths of efficient pentagons (Props. \ref{minmaxangle} and \ref{min-edgelength-pent}). We further restrict the behavior of efficient pentagons in perimeter-minimizing tilings in Proposition \ref{four-angles-don't-tile}, which shows that some efficient pentagons in the tiling must have five angles that tile with the efficient pentagons' angles. 

Definition \ref{def:perimeterRatio} generalizes the concept of the perimeter of a tiling to the planar case by defining the perimeter ratio as the limit supremum of the perimeters of the tiling restricted to increasingly large disks. Lemma \ref{truncation-lemma} shows that the limit infimum of the perimeter to area ratio of tiles completely contained within disks of radius $R$ centered at the origin does not exceed the perimeter ratio of a tiling. Propositions \ref{plane-ext-ratio1} and \ref{plane-ext-stronger-ratio1} generalize our results on the lower bound of the ratio of convex to non-convex pentagons in the general case, and in the special case when all the convex pentagons in the tiling are efficient. Proposition \ref{plane-ext-angles-convex-between} shows that planar tilings by non-convex pentagons and pentagons with angles strictly between $\pi/2$ and $2\pi/3$ have a perimeter ratio higher than that of a Cairo/Prismatic tiling, generalizing Proposition \ref{angles-convex-between}. 
Finally, Proposition \ref{equilateralPentagonMin} finds the perimeter-minimizing unit-area equilateral convex pentagon that tiles the plane monohedrally.

\subsection{Organization}
Section 2 explores tilings of large flat tori by efficient and non-convex pentagons. It provides results restricting the angles and edge-lengths of efficient pentagons and describes particular efficient pentagons of interest. Additionally it limits the ways in which efficient and non-convex pentagons interact in mixed tilings with perimeter less than Cairo/Prismatic, if such tilings exist, and considers efficient and non-convex pentagons outside the context of a tiling. The propositions in Section 2 are used to prove the main results in Sections 3 and 4. 
Section 3 shows that a tiling of a large, flat torus by an efficient pentagon and any number of non-convex quadrilaterals cannot have perimeter less than a Cairo/Prismatic tiling. 
Section 4 shows that a dihedral tiling of a large flat torus by an efficient pentagon and so-called Type 2 non-convex pentagons cannot have perimeter less than a Cairo/Prismatic tiling. 
Section 5 generalizes results on large, flat tori to similar results on the plane by limit arguments. 
Section 6 considers special cases of the main conjecture, such as dihedral tilings by efficient non-convex pentagons, where it may be easier to show that Cairo and Prismatic tilings are perimeter minimizing.
The final appendix section provides the perimeter-minimizing equilateral pentagon that tiles the plane monohedrally.

\subsection{Acknowledgements}
This paper is work of the 2012 ``SMALL'' Geometry Group, an undergraduate research group at Williams College, continued in Martin's thesis \cite{ZaneThesis}. Thanks to our advisor Frank Morgan, for his patience, guidance, and invaluable input. Thanks to Professor William Lenhart for his excellent comments and suggestions. Thanks to Andrew Kelly for contributions to the summer work that laid the groundwork for this paper. Thanks to the National Science Foundation for grants to the Williams College ``SMALL'' Research Experience for Undergraduates, and Williams College for additional funding. Additionally thank you to the Mathematical Association of America (MAA), MIT, the University of Chicago, the University of Texas at Austin, Williams College, and the NSF for a grant to Professor Morgan for funding in support of trips to speak at MathFest 2012, the MAA Northeastern Sectional Meeting at Bridgewater State, the Joint Meetings 2013 in San Diego, and the Texas Undergraduate Geometry and Topology Conference at UT Austin (texTAG).


\setcounter{equation}{0}
\section{Pentagonal Tilings}
\label{penta}
In 2001, Thomas Hales \cite[Thm.1-A]{hales} proved that the regular hexagon provides a most efficient unit-area tiling of the plane. Of course, for triangles and quadrilaterals the perimeter-minimizing tiles are the equilateral triangle and the square. Unfortunately, the regular pentagon doesn't tile. There are, however, two nice pentagons which do tile.

\begin{definitions}
\label{def:CairoPrismaticEfficient}

\emph{While the terms are sometimes used in a broader sense, we define a pentagon as }Cairo \emph{or} Prismatic \emph{if it has three angles of $2\pi/3$, two right angles, nonadjacent or adjacent, respectively, and is circumscribed about a circle, as in Figure \ref{fig:cairoprismatic}. For unit area, both have perimeter $2\sqrt{2+\sqrt{3}} \approx 3.86$.} 

\emph{In this paper, we assume that all tilings by polygons are} edge-to-edge; \emph{ that is, if two tiles are adjacent they meet only along entire edges or at vertices.}

\emph{We say that a unit-area pentagon is} efficient \emph{if it has a perimeter less than or equal to that of a Cairo pentagon's, and that a tiling is} efficient \emph{if it has a perimeter per tile less than half the perimeter of a Cairo pentagon's. Note that a non-convex pentagon can never be efficient because it has more perimeter than a square, the optimal quadrilateral.} \emph{An} efficient vertex \emph{is a vertex in a tiling which is surrounded exclusively by efficient pentagons.}

\emph{Finally, given a sequence of angles $a_i$, we say that an angle $a_j$} tiles \emph{if for some positive integers $m_i$ including $m_j$, $\sum m_i a_i = 2 \pi$.}
\end{definitions}

\begin{remarks}
\emph{Note that an efficient pentagon cannot tile monohedrally. If it did, it would violate Theorem \ref{chungthm}. But an efficient pentagon could have five angles that tile.}

\emph{An efficient pentagon cannot have more than two edges greater than $\sqrt{2}$ because by definition its perimeter is less than a Cairo pentagon's, about 3.86, which is less than $3\sqrt{2}$.}

\emph{While} isoperimetric \emph{tilings by pentagons have been considered only recently \cite{pen11},
there has been extensive research on pentagonal tilings in
general. There are 14 known types of convex
pentagons which tile the plane monohedrally, but no proof that these types form a complete list,
despite notable recent progress by Bagina \cite{bagina11} and Sugimoto and Ogawa \cite{sugi2006}, \cite{sugi1}, \cite{sugi2}, \cite{sugi3}. There is a
complete list for equilateral convex pentagons (\cite{hirsch&hunt}, see also \cite{bagina04}) and apparently for all equilateral pentagons \cite{hirhunt}. These sources provide partial results regarding the properties of convex pentagons which tile, and focus their attention on showing that the known list of 14 types of pentagonal tiles is complete. Hirschhorn and Hunt consider non-convex equilateral pentagons which tile the plane (\cite{hirhunt}), but more general studies of types of non-convex pentagons which tile are absent from the literature, as are any in-depth considerations of tilings by mixtures of convex and non-convex pentagons.}

\emph{Chung et al. \cite[Thm. 3.5]{pen11} proved that Cairo and Prismatic pentagons provide optimal ways to tile the plane using (mixtures of) \emph{convex} pentagons, but were unable to remove the convexity assumption. We conjecture that their results hold without the convexity assumption, and rule out certain tilings with mixtures of convex and non-convex pentagons, though the main conjecture remains open. We begin with the main result from Chung et al.}
\end{remarks}

\begin{theorem} \cite[Thm 3.5]{pen11}
\label{chungthm}
Perimeter-minimizing planar tilings by unit-area convex polygons with at most five sides are given by Cairo and Prismatic tiles.
\end{theorem}

Various times throughout the paper we use the following planar case of a theorem of Lindel\"{o}f (\cite{lind}, see Florian \cite[pp. 174-180]{florian} and Chung et al. \cite[Prop 3.1]{pen11} from before the authors knew about Lindel\"{o}f):

\begin{theorem}[Lindel\"{o}f's Theorem \cite{lind}]
\label{lindelof-lemma}
For $n$ given angles, the $n$-gon circumscribed about a circle is uniquely perimeter minimizing for its area. 
\end{theorem}

Chung et al. give an explicit formula for finding the perimeter of an $n$-gon circumscribed about a circle and add an immediate corollary to the result:

\begin{lemma}\cite[Prop. 3.1]{pen11}
\label{cot-perimeter-lemma}
Scaled to unit area, an $n$-gon with  angles $0 < a_i \leq \pi$ has perimeter greater than or equal to
\begin{equation}
2\sqrt{\sum_{i=1}^n \cot(a_i/2)}  ,
\end{equation}
with equality holding if and only if the $n$-gon is circumscribed about a circle.
For convex $n$-gons, since cotangent is strictly convex up to $\pi/2$, the more nearly equal the angles, the smaller the perimeter.
\end{lemma}

The following proposition follows directly from the above, and will be useful later on in proving our main results.

\begin{proposition}
\label{efficient-pentagon-one-small-angle}
If two angles in a pentagon average less than $\pi/2$ then the pentagon cannot be efficient. If two angles average exactly $\pi/2$, the pentagon is efficient only if it is Cairo or Prismatic.
\end{proposition}

\begin{proof}
Suppose that at least two angles are each less than or equal to $\pi/2$. By Lemma \ref{cot-perimeter-lemma}, the perimeter is uniquely minimized when exactly two angles equal $\pi/2$, the other angles are equal, and the pentagon is circumscribed about a circle, i.e. for Cairo and Prismatic. Therefore the pentagon is not efficient if the average of two angles is less than $\pi/2$, and if the average is equal to $\pi/2$ the pentagon is Cairo Prismatic.
\end{proof}

We begin our analysis of non-convex pentagons, first by categorizing them into two types.

\begin{proposition}
\label{3typesnonconvex}
There are two types of non-convex pentagons, as in Figure \ref{fig:nonconvexfig}:

\begin{enumerate}
\item a non-convex pentagon with one interior angle larger than $\pi$, 
\item a non-convex pentagon with two interior angles (these can be adjacent or non-adjacent) larger than $\pi$ whose average is less than $3\pi/2$,
\end{enumerate}
A unit-area Type 1 pentagon has perimeter greater than a square's (4). A unit-area Type 2 pentagon has perimeter greater than an equilateral triangle's (about 4.559).
\end{proposition}

\begin{proof}
If a pentagon has more than two interior angles larger than $\pi$, then the sum of the interior angles will be greater than $3\pi$, which is a contradiction since the sum of all the interior angles of a pentagon is always $3\pi$. Therefore, either it has one angle larger than $\pi$, as in Case 1, or it has two. If it has two angles larger than $\pi$, and the average of the two angles is greater than $3\pi/2$, then they will sum to more than $3\pi$, which is a contradiction. Hence, the average of the two angles must be less than $3\pi/2$. These large angles are either adjacent or not adjacent, so we have Case 2.

To prove the final statement, just note that taking the convex hull and then scaling down to unit area reduces perimeter, and that the square and equilateral triangle minimize perimeter for given area and number of sides.
\end{proof}

\begin{figure}
		\centering
	\includegraphics[scale=0.4]{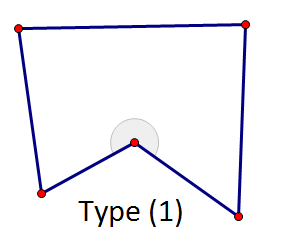}
	\includegraphics[scale=0.4]{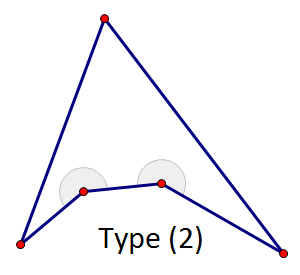}
	\includegraphics[scale=0.4]{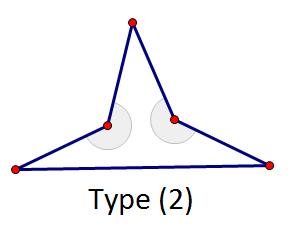}
    \caption{A non-convex pentagon can have one or two interior angles greater than $\pi$.}
    \label{fig:nonconvexfig}
\end{figure}

\begin{remark}
\emph{By Proposition \ref{3typesnonconvex} a unit-area Type 1 non-convex pentagon must have at least one edge with length at least 4/5.}
\end{remark}

We now bound the edge-lengths of unit-area non-convex quadrilaterals and then extend this bound to Type 2 non-convex pentagons.

\begin{proposition}
\label{quad-min-biggest-edge}
The quadrilateral formed by taking a right, isosceles triangle and adding a vertex at the midpoint of the hypotenuse, minimizes longest edge-length for given area among quadrilaterals with an angle measuring $\pi$.
\end{proposition}

\begin{proof}
Given two sides of length at most $a$, a right isosceles triangle with sides of length $a$ maximizes area. The result follows.
\end{proof}

\begin{lemma}
\label{quad-quadrilateral-max-edge}
For a unit-area non-convex quadrilateral, some edge must be greater than or equal to $\sqrt{2}$.
\end{lemma}

\begin{proof}
Assume to the contrary that there existed a unit-area non-convex quadrilateral with four edges less than $\sqrt{2}$. Replacing one of the non-convex angle with a straight line with a vertex in the center and scaling down to unit area would contradict Proposition \ref{quad-min-biggest-edge}.
\end{proof}

\begin{lemma}
\label{ncvx-pent-max-edge}
For a unit-area Type 2 non-convex pentagon, some edge must be greater than or equal to $\sqrt{2}$.
\end{lemma}

\begin{proof}
Assume to the contrary that there existed a unit-area non-convex pentagon with five edges less than $\sqrt{2}$. Replacing the non-convex angle with a straight line with a vertex in the center and scaling down to unit area would contradict Lemma \ref{quad-quadrilateral-max-edge}.
\end{proof}

Chung et al. \cite[Prop. 2.11]{tile11} prove that in a pentagonal tiling of a flat torus with perimeter per tile less than a Prismatic tiling, the ratio of convex to non-convex pentagons is greater than 2.6. We can immediately infer from their proof that the ratio of \emph{efficient} pentagons to non-convex pentagons is greater than 2.6. We further strengthen this result.

\begin{proposition}
\label{ratio1}
Let $T$ be a tiling of a flat torus by unit-area pentagons, with perimeter per tile less than or equal to half the perimeter of a Prismatic pentagon.
Then the fractions $C_1$, $N_1$, and $N_2$ of efficient, Type 1 non-convex, and Type 2 non-convex pentagons in the tiling satisfy $C_1 > 2.6N_1 + 13.4N_2$. 
\end{proposition}

\begin{proof}
We follow a similar proof given by Chung et al. \cite[Prop. 2.11]{tile11}. The perimeters of a regular pentagon, a Cairo/Prismatic pentagon, the unit square, and the unit-area equilateral triangle are $P_0 = 2\sqrt{5} \sqrt[4]{5 - 2\sqrt{5}}$, $P_1 = 2\sqrt{2 + \sqrt{3}}$, $P_2 = 4$ and $P_3 = 3 \sqrt{4/\sqrt{3}}$. Since each edge appears in the perimeter of two tiles, twice the perimeter per tile is at least 
$$
C_1 P_0 + C_2 P_1 + N_1 P_2 + N_2 P_3 \leq P_1 = (C_1 + C_2 + N_1 + N_2)P_1.
$$
Therefore,
$$
C_1 \geq N_1 \frac{P_2 - P_1}{P_1 - P_0} + N_2 \frac{P_3 - P_1}{P_1 - P_0} > 2.6 N_1 + 13.4 N_2.
$$
\end{proof}

Under certain conditions, the convexity hypothesis is easy to rule out, as in the following proposition.

\begin{proposition}
\label{angles-convex-between}
A unit-area tiling of a flat torus by non-convex pentagons and pentagons with angles strictly between $\pi/2$ and $2\pi/3$ has more perimeter per tile than half the perimeter of a Prismatic pentagon.
\end{proposition}

\begin{proof}
Assume, on the contrary, that there exists a unit-area tiling of a flat torus by non-convex pentagons and convex pentagons with angles strictly between $\pi/2$ and $2\pi/3$ which has less perimeter per tile than half the Prismatic pentagon's. By Proposition \ref{ratio1}, the ratio of convex pentagons to non-convex pentagons must be greater than 2.6. Since all the angles of the convex pentagons are strictly between $\pi/2$ and $2\pi/3$, there is at least one non-convex pentagon at each vertex. By definition, a non-convex pentagon has at least one angle greater than $\pi$. Thus at least $1/5$ of the vertices must contain an angle greater than $\pi$. At such vertices there is at most one convex pentagon. At the remaining vertices, there are at most three convex pentagons, because their angles are greater than $\pi/2$. Thus the ratio of convex pentagons to non-convex pentagons is at most $3(4/5)+1(1/5) = 2.6$.  This is a contradiction of Proposition \ref{ratio1}, which says the ratio of convex to non-convex pentagons must be strictly greater than 2.6.
\end{proof}

\begin{remark}
\emph{The reason we need the angles to be strictly between $\pi/2$ and $2\pi/3$ is that our argument depends on having no vertices completely covered by convex pentagons. We deal with other cases separately. Some special cases are easy to eliminate. For example, if a pentagon has two $3\pi/4$ angles and three $\pi/2$ angles, then the perimeter-minimizing pentagon has perimeter equal to about 3.91, which is more than the Prismatic pentagon's.}
\end{remark}

The next few propositions better describe efficient pentagons by bounding their angles and edge-lengths.

\begin{proposition}
\label{minmaxangle}
The interior angles $a_i$ of an efficient pentagon satisfy $80.91^\circ < a_i < 142.29^\circ$.
\end{proposition}
  
\begin{proof}
By Lemma \ref{cot-perimeter-lemma}, it is enough to check the proposition when the smallest angle is 80.92 and the others are equal and, similarly, when the largest angle is 142.29 and the others are equal. At these values, the perimeter is about 3.8638 (to four decimal places), greater than the Prismatic perimeter of about 3.8637. 

\begin{figure}
	\centering
	\includegraphics[scale=0.6]{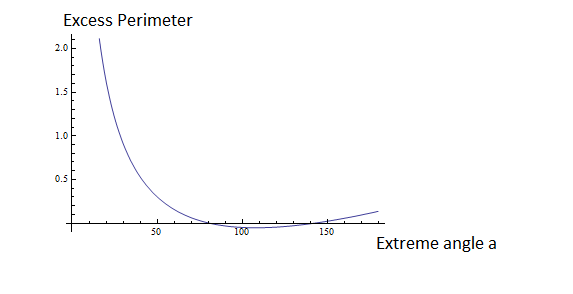}
	\caption{A pentagon with angles far from $108^\circ$ has lots of perimeter.}
	\label{fig:pentsmallangle}
\end{figure}

\end{proof}

\begin{corollary}
\label{convex-vertices-degnot-2}
An efficient vertex in a tiling must have degree equal to three or four.
\end{corollary}

Figure \ref{fig:pentsmallangle} shows the excess perimeter over the Prismatic perimeter for pentagons circumscribed about a circle with one angle $a$ and the other angles equal.

We now provide an alternate proof of a proposition of Chung et al.:

\begin{proposition}
\label{min-edgelength-pent}
\cite[Lem. 3.6]{tile11}
The perimeter-minimizing unit-area pentagon with a given edge-length $e$ is the one inscribed in a circle with one edge of length $e$ and the other four equal.
\end{proposition}

\begin{figure}
		\includegraphics[scale=0.6]{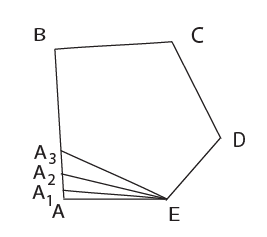}
	\caption{A method for scaling pentagon ABCDE down to unit-area, keeping edge DE fixed - replacing line segment $AE$ with $A_3E$ and decreasing perimeter.}
	\label{fig:MinGivenEdge}
\end{figure}

\begin{proof}
Let $ABCDE$ be a perimeter-minimizing pentagon with $DE$ of given length $l$. 
It is well known that for given edge-lengths, the $n$-gon inscribed in a circle uniquely maximizes area. If the $ABCDE$ were not inscribed in a circle, we can increase its area by inscribing it, keeping the edges, and thus the perimeter, fixed. We then scale back down to unit area keeping $DE$ fixed. One method to perform this is shown in Figure \ref{fig:MinGivenEdge} -- we simply replace line segment $EA$ with line segment $EA_i$ until we arrive at unit area. This decreases the perimeter, contradicting that $ABCDE$ is perimeter-minimizing. Therefore $ABCDE$ must be inscribed in a circle. 

Now we show that the other four edges must be equal. Suppose two such adjacent edges, say $AB$ and $BC$, have different lengths. We can replace triangle $ABC$ with an isosceles triangle $AB'C$ such that $|AB'| + |B'C| = |AB| + |AC|$, but $AB'C$ has greater area than $ABC$. We do not need to worry about $AB'C$ overlapping or bumping into another edge or vertex. Scaling down to unit-area but keeping $DE$ fixed (using the method in Figure \ref{fig:MinGivenEdge}) we have decreased the perimeter, a contradiction as $ABCDE$ is perimeter minimizing. Therefore the other four edge-lengths must be equal, and the proposition follows.
\end{proof}

\begin{lemma}
\label{quad-pentagon-min-edge}
A unit-area efficient pentagon cannot have an edge greater than 1.081 or less than .4073.
\end{lemma}

\begin{proof}
By Proposition \ref{min-edgelength-pent},  for a given edge-length $e$, the pentagon $X$ inscribed in a circle with four equal sides and one side $e$ minimizes perimeter. Let $r$ be the radius of the circle and $\alpha$ be the angle between rays from the center of the circle to adjacent vertices of the four equal sides. Then the perimeter and area formulae for $P$ are:
$$
P = 8r \sin(\alpha/2) + 2r \sin(2\alpha),
$$
$$
A = 2r^2 \sin(\alpha) - (1/2)r^2 \sin(4\alpha).
$$

By assumption the area is one, and the perimeter of a Cairo/Prismatic pentagon is approximately $3.86$, our value for $P$.
Then solving for $r^2$ we get
$$
r^2 = \frac{3.86^2}{(8 \sin(\alpha/2) + 2 \sin(2\alpha))^2}
$$
and therefore
$$\alpha \approx 62.8942, 81.0705. 
$$
From this we conclude that the pentagon is efficient only when
$$
62.8941 < \alpha < 81.0706.
$$
Then $e = 2r\sin(2\alpha)$,
which implies that $.4073 < e < 1.081$.
\end{proof}

Recall that an angle $a_j$ tiles if given a sequence of angles $a_i$, there exist some positive integers $m_i$ including $m_j$ such that $\sum m_i a_i = 2 \pi$. Then for tilings by efficient pentagons and Type 2 non-convex pentagons we have the following:

\begin{proposition}
\label{four-angles-don't-tile}
Consider a unit-area tiling of a flat torus by efficient pentagons and Type 2 non-convex pentagons. Assume that each efficient pentagon has at most four angles which tile with efficient pentagons. Then the tiling has more perimeter per tile than half the perimeter of a Prismatic pentagon.
\end{proposition}

\begin{proof}
Because the efficient pentagon has at most four angles which tile, it cannot be surrounded entirely by efficient pentagons. Therefore the ratio of efficient pentagons to non-convex pentagons is at most the maximum number of efficient pentagons which can surround a non-convex pentagon.

At the three angles of the non-convex pentagon which are less than $\pi$, there are at most four efficient pentagons. If there were five or more, then the angles of the efficient pentagon would be too small, in violation of Proposition \ref{minmaxangle}. By the same logic, there are at most two efficient pentagons at the two angles in the non-convex pentagon which is greater than $\pi$. Since the tiling is edge-to-edge, five of the efficient pentagons surrounding the non-convex pentagon appear at two vertices, and we must avoid double counting these. So the total number of efficient pentagons surrounding the non-convex pentagon is $4(3) + 2(2) - 5 = 11$. Then the ratio of efficient pentagons to non-convex pentagons is at most eleven to one. Therefore by Proposition \ref{ratio1} the tiling has more perimeter per tile than half that of a Prismatic pentagon.
\end{proof}

We have a similar result for certain Type 1 non-convex pentagons, though the proposition does not hold for all Type 1 non-convex pentagons.

\begin{proposition}
Consider a tiling of a flat torus by efficient pentagons and Type 1 non-convex pentagons with perimeter greater than 4.537. Assume that each efficient pentagon has at most four angles which tile with efficient pentagons. Then the tiling has more perimeter per tile than half the perimeter of a Prismatic pentagon.
\end{proposition}

\begin{proof}
Because the efficient pentagon has at most four angles which tile, it cannot be surrounded entirely by efficient pentagons. Therefore the ratio of efficient pentagons to non-convex pentagons is at most the maximum number of efficient pentagons which can surround a non-convex pentagon.

At the four angles of the convex pentagon which are less than $\pi$, there are at most four efficient pentagons, otherwise the efficient pentagons would contradict Proposition \ref{minmaxangle}. By the same logic, there are at most two efficient pentagons at the angle which is greater than $\pi$. Since the tiling is edge-to-edge, five of the efficient pentagons surrounding the non-convex pentagon will appear at two vertices, and we need to avoid double counting these. Then the ratio of efficient pentagons to non-convex pentagons is $4(4) + 2 - 5 = 13.$

Assume such a tiling had perimeter per tile less than half that of a Cairo pentagon. Let $P_0$, $P_1$, and $P_2$ denote the perimeter of a regular pentagon, a Cairo/Prismatic pentagon, and 4.537. If there are $m$ convex pentagons and $n$ non-convex pentagons then by hypothesis
$$
mP_0 + nP_2 < (m+n)P_2,
$$
which implies
$$
\frac{m}{n} > \frac{P_1 - P_2}{P_0 - P_1} > 13,
$$
which is a contradiction. 
\end{proof}

The following will be useful in proving our main results, as it limits the perimeter of a certain class of efficient pentagons.

\begin{figure}
		\includegraphics[scale=0.4]{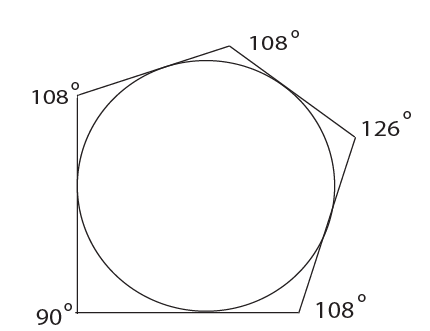}
	\caption{The perimeter-minimizing pentagon, with five angles that tile, at least one angle of which can tile a degree-four vertex.}
	\label{fig:FiveTileMinimizer}
\end{figure}

\begin{proposition}
\label{pminimizing5tileDeg4Vert}
The perimeter-minimizing pentagon $P$ with five angles that tile, at least one angle of which can tile a degree-four vertex, is the one circumscribed about a circle with one $90^\circ$ angle, three $108^\circ$ angles, and one $126^\circ$ angle, as in Figure \ref{fig:FiveTileMinimizer}. The perimeter is approximately $3.8414$. 
\end{proposition}

\begin{proof}
First note that such a perimeter-minimizing pentagon must exist. Consider a sequence of such pentagons $P_n$ with perimeter converging to the infimum. We may assume that the pentagons are convex (see Definitions \ref{def:CairoPrismaticEfficient}). By standard compactness, the desired limit exists.  

Since by hypothesis the pentagon can tile a degree-four vertex, we have the following cases:
\newline
\newline
\noindent \textbf{Case 1: Exactly one angle tiles a degree-four vertex.}
\newline

Note that this implies there must be one $90^\circ$ angle and that the other four angles must tile a degree-three vertex. If the other four angles tiled a vertex of degree greater than or equal to five the pentagon would not be efficient by Corollary \ref{convex-vertices-degnot-2}. We proceed to cover all the cases with regards to which of the four non-right angles are equal to one another.

First note that these four angles cannot all be equal, as they would equal $(112.5^\circ)$, which does not tile with itself or $90^\circ$. 

Next suppose three of the other angles, $x$, are equal and one, $y$, is different. Then 
$$
3x + y = 450^\circ,
$$ 
as the angles of a pentagon sum to $540^\circ$. Note that as $x$ and $y$ are not equal to each other, they cannot both equal $120^\circ$. This implies they must tile together in some way. So either
$2x + y = 360^\circ$ or $x + 2y = 360^\circ$. 
In the first case $y = 180^\circ$, which violates Proposition \ref{minmaxangle}. The second has has $x = 108^\circ$ and $y = 126^\circ$, and is the pentagon $P$.

Next suppose there are two pairs of equal angles, denoted $x$ and $y$. Then $2x + 2y = 450^\circ$ and as before either $2x + y = 360^\circ$ or $x + 2y = 360^\circ$. 
In either case there are three $90^\circ$ angles and two $135^\circ$ angles. This pentagon has perimeter approximately $3.9132$, greater than the perimeter of $P$.

Next suppose two angles are equal and two are unequal. Let $x$ denote the equal angles and $y$ and $z$ the unequal ones. Then $2x + y + z = 450^\circ$. First note that if $x$, $y$, or $z$ tile with the $90^\circ$ angle, they will equal either $180^\circ$ (if there are two $90^\circ$ angles tiling with them) or $135^\circ$ (if there is one $90^\circ$ and two copies of $x$). The first case is not efficient by Proposition \ref{minmaxangle}. The second case does not have perimeter less than $P$ by Lemma \ref{cot-perimeter-lemma}. So $x$, $y$ , and $z$ tile together. If $x + y + z = 360^\circ$, this implies that $x = 90^\circ$, a contradiction that exactly one angle tiles a degree-four vertex. 
If $x = 120^\circ$ and it tiles with $y$ or $z$ the optimal pentagon  of this form is Cairo, which has perimeter greater than $P$. If $x$ doesn't tile with $y$ or $z$ then $y$ and $z$ tile together. Without loss of generality, assume $2y + z = 360^\circ$. Then since we also know $y + z = 210^\circ$ we conclude that $y$ must be be $150^\circ$, which means the pentagon is not efficient by Proposition \ref{minmaxangle}.  

So the possible cases for how $x$ tiles are either $2x + y = 360^\circ$ or $x + 2y = 360^\circ$, with $x \not= 120^\circ$ (if $x$ tiles with $z$ just switch the labels on $y$ and $z$). First note that if $2x + y = 360^\circ$ and $2z + x = 360^\circ$, $x = 180^\circ$ which contradicts Proposition \ref{minmaxangle}. Now if $2x + y = 360^\circ$ then additionally either $2y + z = 360^\circ$ or $2z + y = 360^\circ$. 
Since we know $2x + y + z = 450^\circ$, we can solve for $x, y$ and $z$. In both cases $y = 180^\circ$, which means the pentagon is not efficient by Proposition \ref{minmaxangle}.

Similarly if $x + 2y = 360^\circ$ either $2y + z = 360^\circ$ or $2z + y = 360^\circ$. Then we have two possible pentagons: one has $x = 108^\circ$, $y = 126^\circ$, and $z =108^\circ$, and so is just $P$, and the second has $x= 720^\circ/7$, $y= 900^\circ/7$, and $z = 810^\circ/7$, and perimeter greater than $3.849$. This completes the case when two angles are equal and two are unequal.

The final subcase is when all four angles are unequal. Then $w + x + y + z = 450^\circ$. We have two options with regard to how angle $x$ tiles - either $x + y + z = 360^\circ$ or $2x + y = 360^\circ$ (switching the labels on $x$ and $y$ allows us to assume this). 
If $x + y + z = 360^\circ$, then $w = 90^\circ$. As $x, y$ and $z$ are not equal, this implies the pentagon is not efficient.

Assume that 
$$
2x + y = 360^\circ.
$$
Since $w + x + y + z = 450^\circ$, we can express $x = z+w - 90^\circ$ and $y = 540^\circ - 2(z+w)$. Substituting in $s$ for $z + w$, we know that the perimeter-minimizing pentagon satisfying these requirements has angles measuring $90^\circ, s - 90^\circ, 540^\circ - 2s, s/2, s/2$. 
Note that here we assume $w = z$, and do not consider tiling by these two angles. While this violates the conditions on these pentagons, it provides a lower bound for pentagons of this general form. 

The only case when these pentagons have perimeter less than $3.8414$ is when $ 210^\circ < s < 216^\circ$, by applying Lemma \ref{cot-perimeter-lemma} and using Mathematica to calculate the perimeter for all $s$ in the allowable range. So $210^\circ < z+w < 216^\circ$. 
If $2z + w = 360^\circ$ or $z + 2w = 360^\circ$ then either $w$ or $z$ is greater than $144^\circ$, so the pentagon will not be efficient by Proposition \ref{minmaxangle}. Also note that $2x + z = 360^\circ$ and $2z + y = 360^\circ$ are not allowable, as $z$ is not equal to $x$ or $y$. If $z + 90^\circ + x = 90^\circ$ or $z + y + 90 = 360^\circ$, this implies that $x$ or $y$ equals $270^\circ - w$. 
The largest an angle can be in a pentagon with one $90^\circ$ angle and perimeter less than 3.8414 is $127^\circ$. This implies $x$ or $y$ is at least $143^\circ$, which implies that the pentagon is not efficient by Proposition \ref{minmaxangle}.

So either $2y + z = 360^\circ$ or $2z + x = 360^\circ$. Then given that $x + y + z + w = 450^\circ$ and $2x + y = 360^\circ$, we have three equations and four variables. Putting everything in terms of $w$, we get two pentagons. One has angles $x = (450^\circ-w)/2$, $y = 60^\circ+2w/3$ and $z = 240^\circ-4w/3$. The other has angles $x = 2(90^\circ+w)/3$, $y = 240^\circ-4w/3$ and $z = 150^\circ-w/3$. 
Using Lemma \ref{cot-perimeter-lemma} and Mathematica, we can plot the minimum perimeter in terms of $w$. We observe that for no value of $w$ will the perimeter be less than $3.8414$.

Then we have seen that no matter how the other four angles interact, pentagons with exactly one angle which can tile a degree-four vertex have perimeter greater than or equal to $P$.
\newline
\newline
\noindent \textbf{Case 2: Exactly two angles tile a degree-four vertex.}
\newline

If they are of the form $x + y = 180^\circ$, by Lemma \ref{efficient-pentagon-one-small-angle} the perimeter will be at least that of a Cairo pentagon's, which is greater than the perimeter of $P$. If they are of the form $x + 3y = 360^\circ$ and $x < y$ then the average of the two must be less than $90^\circ$, so the perimeter will be greater than that of a Cairo pentagon. If $y < x$ then for given $y$ the pentagon with least perimeter has angles $y, 360^\circ - 3y, (180^\circ + 2y)/3$. Then for values of $y$ from $80.92^\circ$ to $90^\circ$ - the only allowable ranges for $y$ - we can see graphically that the perimeter is greater than $3.84143$ using Mathematica. 

These are the only possible cases, therefore the perimeter-minimizing pentagon with exactly two angles that tile a degree-four vertex perimeter greater than or equal to $3.8414$.
\newline
\newline
\noindent \textbf{Case 3: Exactly three angles tile a degree-four vertex.}
\newline

In this case, we know that it is the case that $2x + y + z = 360$. Thus $x + y + z$ is at most $279.08^\circ$ by Proposition \ref{minmaxangle}. So the other two angles, $a$ and $b$, average at least $260.92^\circ$. The perimeter is minimized when $a = b$, i.e. when there are two angles measuring at least $130.46^\circ$. So the perimeter is at least $3.88$ by Lemma \ref{cot-perimeter-lemma}, and the pentagon will not be efficient.
\newline
\newline
\noindent \textbf{Case 4: Exactly four angles tile a degree-four vertex.}
\newline

Then the fifth angle must equal $180^\circ$, which violates Proposition \ref{minmaxangle}.
\end{proof}

\begin{proposition}
\label{type1-cvx-best}
The least-perimeter unit-area pentagon with angles not strictly between $\pi/2$ and $2\pi/3$ is circumscribed about a circle with one $2\pi/3$ angle and four $7\pi/12$ angles. It has perimeter approximately $3.819$.
\end{proposition}

\begin{proof}
This follows directly from Lemma \ref{cot-perimeter-lemma}: the more nearly equal the angles, the smaller the perimeter. Excluding pentagons with angles between $\pi/2$ and $2\pi/3$, we are left with the following two pentagnos as the optimal choices: a pentagon circumscribed about a circle with one $2\pi/3$ angle and four $7\pi/12$ angles, and a pentagon circumscribed about a circle with on $\pi/2$ angle and four $5\pi/8$ angles. Using Lemma \ref{cot-perimeter-lemma} to calculate the perimeter yields the result.
\end{proof}


\section{Tilings of Pentagons and Quadrilaterals}

We now turn our attention to tilings by pentagons and quadrilaterals. In particular we consider tilings with one efficient pentagon and any number of non-convex quadrilaterals, though some of our results hold for tilings by efficient pentagons and non-convex quadrilaterals. Tilings with efficient pentagons, non-convex quadrilaterals, and non-efficient convex pentagons or convex quadrilaterals remain relatively unstudied, and we do not know whether our results might generalize to that case.

Recall from Proposition \ref{type1-cvx-best} we have the following:

\begin{repproposition}{type1-cvx-best}
The least-perimeter unit-area pentagon with angles not strictly between $\pi/2$ and $2\pi/3$ is circumscribed about a circle with one $2\pi/3$ angle and four $7\pi/12$ angles. It has perimeter approximately $3.819$.
\end{repproposition}

\begin{proposition}
\label{perimeterMinimizingDegreeFourPentagon}
The perimeter-minimizing unit-area pentagon which can tile a degree-four efficient vertex has perimeter at least 3.8328.
\end{proposition}

\begin{proof}
In order to tile a degree-four vertex, the perimeter-minimizing pentagon must have at least one angle measuring at most $90^\circ$. By Lemma \ref{cot-perimeter-lemma}, the pentagon circumscribed about a circle with one angle measuring $90^\circ$ and four angles measuring $112.5^\circ$ is perimeter-minimizing, with perimeter greater than 3.8328.
\end{proof}

Recall from Section 2 we have Proposition \ref{pminimizing5tileDeg4Vert}:

\begin{repproposition}{pminimizing5tileDeg4Vert}
The perimeter-minimizing pentagon with five angles that tile and at least one angle which can tile a degree-four vertex is the one circumscribed about a circle with one $90^\circ$ angle, three $108^\circ$ angle, and one $126^\circ$ angle. This has perimeter approximately $3.8414$. 
\end{repproposition}

\begin{proposition}
\label{nonconvex-quad-not-surrounded}
In unit-area tiling by non-convex quadrilaterals and a single type of efficient pentagon, a non-convex quadrilateral cannot be surrounded by efficient pentagons.
\end{proposition}

\begin{proof}
By Lemma \ref{quad-quadrilateral-max-edge}, some edge of each non-convex quadrilateral exceeds $\sqrt{2} > 1.41$, but by Lemma \ref{quad-pentagon-min-edge}, the longest edge in an efficient pentagon is less than $1.081$.
\end{proof}

The following proof is loosely based on a previous result by Chung et al. \cite[Lem. 3.6]{tile11}, who demonstrate the perimeter-minimizing pentagon with a given edge-length. We adapt their proof to find the perimeter-minimizing triangle with a given edge-length.

\begin{proposition}
\label{EdgeMinTriangle}
The perimeter-minimizing triangle with given edge-length $e$ is an isosceles triangle with base $e$.
\end{proposition}

\begin{proof}
It is well known that for given edge-lengths, $e_i$, $i = 1, 2, \ldots, n$, the perimeter-minimizing $n$-gon is the one inscribed in a circle. The area is given by
$$
\frac{1}{2}r^2\sum_{i = 1}^3 \sin \theta_i,
$$
where $\theta_i$ is the center angle corresponding to the $e_i$. Since sine is concave down from 0 to $\pi$, for a fixed perimeter the area will be maximized when the angles are equal. Therefore given one edge, the perimeter is minimized when the other two edges are equal.
\end{proof}

\begin{proposition}
\label{quadVertexBound}
For any tiling, any $m$ quadrilateral tiles have at least $m$ vertices.
\end{proposition}

\begin{proof}
The angles of a quadrilateral sum to $2\pi$, and in a tiling a vertex measures exactly $2\pi$. Therefore each quadrilateral contributes one vertex, so the $m$ quadrilaterals contribute $m$ vertices (or portions of more than $m$ vertices).
\end{proof}

\begin{proposition}
\label{gendihedralSomeDeg4Verts}
In a tiling of a flat torus by efficient pentagons and non-convex quadrilaterals, if the ratio of efficient pentagon to non-convex quadrilaterals exceeds $14$ then the tiling will have at least one degree-four efficient vertex. 
\end{proposition}

\begin{proof}
Assume a tiling of a flat torus by $n$ efficient pentagons and $m$ non-convex quadrilaterals. The area of the torus is $n + m$, so by the Euler characteristic formula there are $3n/2 + m$ vertices. By Proposition \ref{nonconvex-quad-not-surrounded}, the non-convex quadrilaterals cannot be surrounded entirely by efficient pentagons because at least one edge in each non-convex quadrilateral is too long to tile with any efficient pentagons. Therefore each non-convex quadrilateral shares an edge with at least one other non-convex quadrilateral, so there are at most 6 inefficient vertices for each pair of non-convex quadrilaterals. By Proposition \ref{minmaxangle}, there are at most $2(2) + 4(4) - 6 = 14$ efficient pentagons (subtracting six because we assume the tiling is edge-to-edge) surrounding each pair of non-convex quadrilaterals. Let $k_3$ and $k_4$ be the number of degree three and degree-four efficient vertices. Then counting each efficient vertex as one-fifth of a pentagon, we know
$$
(3/5)k_3 + (4/5)k_4 \geq n - (14/5)(m/2).
$$
Additionally
$$
3n/2 + m - m = 3n/2 \geq k_3 + k_4.
$$
as by Proposition \ref{quadVertexBound} the $m$ non-convex quadrilaterals contribute at least $m$ inefficient vertices. These bounds imply $k_4 \geq n/2 - 7m.$ Then $k_4 \geq n/2 - 7m > 0$ when $n > 14m$.
\end{proof}

\begin{proposition}
\label{genquad-bound-four-angles-don't-tile}
Consider a unit-area tiling of a flat torus by efficient pentagons and non-convex quadrilaterals. Assume that each efficient pentagon has at most four angles which tile with the efficient pentagon's angles. Then the tiling has a ratio of convex to non-convex pentagons less than or equal to 10.
\end{proposition}

\begin{proof}
Because the efficient pentagon has at most four angles which tile, it cannot be surrounded entirely by efficient pentagons. Therefore the ratio of efficient pentagons to non-convex quadrilaterals is at most the maximum number of efficient pentagons which can surround a non-convex quadrilateral.

At the three angles of the non-convex quadrilateral which are less than $\pi$, there are at most four efficient pentagons. If there were five or more, then the angles of the efficient pentagon would be too small, in violation of Proposition \ref{minmaxangle}. By the same logic, there are at most two efficient pentagons at the angle in the non-convex quadrilateral which is greater than $\pi$. Since the tiling is edge-to-edge, four of the efficient pentagons surrounding the non-convex quadrilateral appear at two vertices, and we must avoid double counting these. So the total number of efficient pentagons surrounding the non-convex pentagon is $4(3) + 2(1) - 4 = 10$. Then the ratio of efficient pentagons to non-convex quadrilaterals is at most ten to one.
\end{proof}

We adapt Proposition \ref{ratio1} to the case of dihedral tilings with quadrilaterals and pentagons.

\begin{proposition}
\label{generalQuadRatio}
In a tiling $T$ of a flat torus by a unit-area efficient pentagon and non-convex quadrilaterals, with perimeter per tile less than or equal to half the perimeter of a Prismatic pentagon, the ratio of efficient pentagons to non-convex quadrilaterals is greater than $31.1753$.
\end{proposition}

\begin{proof}
We adapt a proof given by Chung et al. \cite[Prop. 2.11]{tile11} and Proposition \ref{ratio1}. By Proposition \ref{angles-convex-between}, $T$ cannot have an efficient pentagon with angles strictly between $\pi/2$ and $2\pi/3$. 
Then by Proposition \ref{type1-cvx-best}, the least-perimeter unit-area pentagon with angles not strictly between $\pi/2$ and $2\pi/3$ has perimeter, $P_0$, greater than $3.819$. The perimeter of a Cairo/Prismatic pentagon is $P_1 = 2\sqrt{2+\sqrt{3}} < 3.86$. The convex pentagons have perimeter at least $P_0$ by definition, and the non-convex quadrilaterals have perimeter $P_2$ greater than an equilateral triangle ($3\sqrt{4 / \sqrt{3}}$). 

Let $n$ and $m$ denote the number of efficient pentagons and non-convex quadrilaterals. By hypothesis, 
$$ nP_0 + mP_2 < (n+m)P_1. $$
Therefore
$$ n/m > \frac{P_2-P_1}{P_1-P_0} \approx 15.554.$$

Proposition \ref{gendihedralSomeDeg4Verts} implies that the tiling must contain at least one degree-four efficient vertex, and Propostion \ref{genquad-bound-four-angles-don't-tile} implies that the efficient pentagon must have five angles which tile. Therefore by Proposition \ref{pminimizing5tileDeg4Vert}, the efficient pentagon cannot have perimeter exceeding $P_0' = 3.8414$. Substituting $P_0'$ in for $P_0$ yields $n/m > 31.1753$. 
\end{proof}

\begin{proposition}
\label{gen-quad-four-angles-don't-tile}
Consider a unit-area tiling of a flat torus by an efficient pentagon and non-convex quadrilaterals. Assume that each efficient pentagon has at most four angles which tile with efficient pentagon's angles. Then the tiling has more perimeter per tile than half the perimeter of a Prismatic pentagon.
\end{proposition}

\begin{proof}
This follows immediately from Propositions \ref{generalQuadRatio} and \ref{genquad-bound-four-angles-don't-tile}.
\end{proof}

\begin{proposition}
\label{quadCan'tTileDeg3}
Let $P$ be a efficient pentagon with perimeter less than $3.8574$ and let $s$ and $a$ be angles in $P$ with the following properties:
\begin{enumerate}
\item All five angles of $P$ tile;
\item $s$ is strictly less than $90^\circ$;
\item $3s + a = 360$.
\end{enumerate}
Then $a$ cannot tile a degree three vertex with the angles in $P$.
\end{proposition}

\begin{proof}
As $P$ is efficient, by Proposition \ref{minmaxangle}, $80.92^\circ < s < 90^\circ$. From the given we know $a = 360-3s$. Let $x, y$ be two additional angles in $P$. Suppose that $a$ did tile a degree three vertex with the angles in $P$. Then there are five case, corresponding to the five ways $a$ can be combined with itself, $s$, and $x$ and $y$.
\begin{enumerate}
\item $a + 2s = 360$;
\item $a + s + x = 360$;
\item $a + x + y = 360$;
\item $a + 2x = 360$;
\item $2a + x = 360$.
\end{enumerate}

Case 1 can be easily ruled out, as $a + 3s = 360$ and $a + 2s = 360$ are true only when $s$ is zero and $a$ is $360^\circ$, which obviously violates the properties of $P$.

In Case 2, $360 - 2s + x = 360$ and therefore $x = 2s$. So $161.84 < x$, which by Proposition \ref{minmaxangle} implies that $P$ is not efficient.

By Lemma \ref{cot-perimeter-lemma}, Case 3 reduces to Case 4 as the perimeter will be lowest when $x = y$. Then we can solve for $x$ in terms of $s$ to get $x = 3/2s$ and therefore $540 - (s  + a + x) = (180 + s/2)/2$. Using Mathematica and Lemma \ref{cot-perimeter-lemma} to calculate the minimum perimeter of $P$ in terms of $s$ yields a perimeter greater than 3.8574.

In Case 5, $x = 6s - 360$. So $540 - (s + a + x)$ is $540-4s$, which means the remaining two angles in the pentagon average $(540-4s)/2$ (with perimeter minimized when the two are equal to the average.) Using Lemma \ref{cot-perimeter-lemma} and Mathematica to calculate the minimum perimeter of $P$ in terms of $s$ implies $P$ will never be efficient in this case. Therefore if the perimeter of $P$ is less than $3.8574$, none of these five cases are possible.
\end{proof}

Recall from Proposition \ref{efficient-pentagon-one-small-angle} we have the following:

\begin{repproposition}{efficient-pentagon-one-small-angle}
If two angles in a pentagon average less than $\pi/2$ then the pentagon cannot be efficient. If two angles average exactly $\pi/2$, the pentagon is efficient only if it is Cairo or Prismatic.
\end{repproposition}

\begin{theorem}
\label{oneeffn&mTorus}
A unit-area tiling of a flat torus by an efficient pentagon and non-convex quadrilaterals cannot have perimeter per tile less than half the perimeter per tile of a Cairo/Prismatic tiling.
\end{theorem}

\begin{proof}
Assume there exists a unit-area tiling of a flat torus by $n$ efficient pentagons and $m$ non-convex quadrilaterals with perimeter per tile less than a Cairo/Prismatic tiling. Note that $m \not= 0$, otherwise the tiling would contradict Theorem \ref{chungthm}. The area of the torus is $n + m$, so by the Euler characteristic formula there are $3n/2 + m$ vertices. By Proposition \ref{nonconvex-quad-not-surrounded}, a non-convex quadrilateral cannot be surrounded entirely by efficient pentagons because at least one edge in the non-convex quadrilateral is too long. Therefore each non-convex quadrilateral shares an edge with at least one other non-convex quadrilateral, so there are at most six inefficient vertices for each pair of non-convex quadrilaterals. Then we may assume that there are at most $3m$ inefficient vertices, so the number of efficient vertices is at least 
$$
\frac{3n}{2} + m - 3m = \frac{3n}{2} - 2m.
$$
Let $k_3$ and $k_4$ be the number of degree-three and degree-four efficient vertices. By Corollary \ref{convex-vertices-degnot-2}, these are the only two types of efficient vertices which can appear in the tiling. Therefore
$$
k_3 + k_4 \geq \frac{3n}{2} - 2m.
$$
Additionally, since there are $n$ efficient pentagons, considering each vertex as a fifth of a pentagon we conclude
$$
\frac{3}{5} k_3 + \frac{4}{5} k_4 \leq n.
$$
Thus $k_3 \geq n - 8m$ and $k_4 \leq n/2 + 6m$.

Now by Proposition \ref{generalQuadRatio}, it must be the case that $n > 31.1753m$, as the tiling has perimeter per tile less than a Cairo pentagon's. Therefore by Proposition \ref{gendihedralSomeDeg4Verts}, there exists at least one efficient vertex of degree-four in the tiling. So there must be at least one angle, $s$, in the efficient pentagon which measures less than or equal to $90^\circ$. 

\begin{figure}
		\centering
	\includegraphics[scale=0.5]{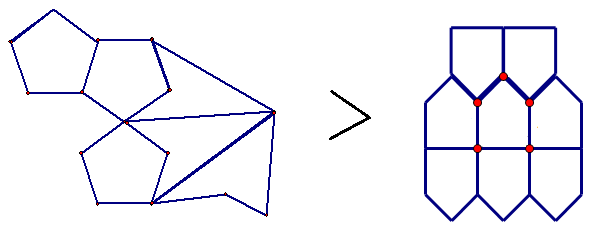}
    \caption{A tiling with an efficient pentagon and non-convex quadrilaterals is always worse than a Prismatic tiling.}
    \label{fig:nonconvexfig}
\end{figure}

Suppose $s$ is the only angle which tiled an efficient vertex of degree four. Then $s = 90^\circ$. Therefore at the large angles of the non-convex quadrilateral there can be at most one efficient angle, and at small angles there can be at most three. As the non-convex quadrilaterals are paired because of their edge-lengths, there are at most $2(1) + 4(3) - 6 = 8$ efficient pentagons (subtracting six because we assume the tiling is edge-to-edge) per pair of non-convex quadrilaterals. Therefore counting each efficient pentagon as one-fifth of a vertex, we know
$$
(3/5)k_3 + (4/5)k_4 \geq n - (8/5)(m/2).
$$
Additionally
$$
3n/2 + m \geq k_3 + k_4,
$$
as there cannot be more efficient vertices than total vertices in the tiling. We can use these two inequalities to determine that $k_4 \geq n/2 - 7m.$ 

We will show this implies there are too many $s$ angles in the tiling. As $s$ is the only angle that can tile a degree-four efficient vertex, there are at least $4k_4 \geq 2n - 28m$ $s$ angles. But as there cannot be more than $n$ $s$ angles, we have that $n \geq 4k_4$, and so $n \geq 2n-28$. Then $28m \geq n$. But since $n > 31.1753m$ by Proposition \ref{generalQuadRatio} this inequality is false, i.e. there are more than $n$ $s$ angles. As this is a contradiction, there must be at least two angles which tile a degree-four vertex.

Next we show there cannot be more than one angle other than $s$ which tiles a degree four efficient vertex. Let $a$, $b$, and $c$ be three angles in the efficient pentagon different than $s$. If $a + b + c + s = 360^\circ$ then the four angles average $90^\circ$ and the pentagon is not efficient by Proposition \ref{efficient-pentagon-one-small-angle}. If $2s + a + b = 360^\circ$, then by Lemma \ref{cot-perimeter-lemma} the perimeter is minimized when two angles measure $(360 - 2s)/2$, two measure $(180 + s)/2$, and one measures $s$. By definition $s \leq 90^\circ$ and by Proposition \ref{minmaxangle} $s > 80.92^\circ$; it follows from Lemma \ref{cot-perimeter-lemma} that the pentagon is not efficient. The case when $2a + s + b = 360^\circ$ is similar: just replace $s$ with $a$ and let $a$ range from $80.92^\circ$ to $142.29^\circ$ by Proposition \ref{minmaxangle}. As before, the pentagon is not efficient in this case. Therefore only one angle, say $a$, can tile a degree four vertex with $s$.

If there are two $s$ and two $a$ angles, then they average exactly $\pi/2$. Since the pentagon is efficient, by Lemma \ref{cot-perimeter-lemma} and Proposition \ref{efficient-pentagon-one-small-angle} it must be a Cairo or Prismatic pentagon, and the proposition follows as $m \not= 0$. 

Therefore there are either three $s$ or three $a$ angles at a degree-four efficient vertex. First suppose there are three $a$ angles and one $s$ angle. Then $3a + s = 360$. By Proposition \ref{minmaxangle}, $s$ must be greater than $80.92^\circ$, and by hypothesis less than or equal to $90^\circ$. It follows that the average of $a$ and $s$ is between $86.973^\circ$ and $90^\circ$. Then by Proposition \ref{efficient-pentagon-one-small-angle}, the efficient pentagon must be Cairo or Prismatic and the proposition follows as $m \not= 0$. Therefore it must be the case that $3s + a = 360$. 

Because the tiling is there is only one type of efficient pentagon, there are exactly $n$ $s$ angles. At all $k_4$ vertices there will be three $s$ angles. By Proposition \ref{minmaxangle}, there are at most $2(2) + 4(4) - 6 = 14$ efficient pentagons surrounding each pair of non-convex quadrilaterals. Therefore counting each efficient vertex as one-fifth of a pentagon, we know
$$
(3/5)k_3 + (4/5)k_4 \geq n - (14/5)(m/2).
$$
Additionally
$$
3n/2 + m \geq k_3 + k_4
$$
which we can use to determine that $k_4 \geq n/2 - 10m.$

Since there are three $s$ angles at each $k_4$ vertex, there are at least $3k_4 \geq 3n/2 - 30m$ $s$ angles. But as there cannot be more than $n$ $s$ angles, we have that $n \geq 3k_4$, and so $n \geq 3n/2 - 30m$, which means $n \leq 60m$.

In order to have a ratio of efficient pentagons to non-convex quadrilaterals less than or equal to 60, the efficient pentagon must have perimeter less than $3.8458$. Here it is convenient to note that since the tiling has perimeter per tile less than Cairo/Prismatic, by Proposition \ref{genquad-bound-four-angles-don't-tile} the efficient pentagon must have five angles which tile. Thus the efficient pentagon satisfies the hypothesis of Proposition \ref{quadCan'tTileDeg3}, so it will not be the case that $a$ tiles an efficient vertex of degree three. Therefore $a$ appears only at degree-four efficient vertices and inefficient vertices.

There will be at most one $a$ angle at each degree-four vertex, and at most three at each non-efficient vertex, as $a$ is greater than $90^\circ$. Since $k_4 \leq n/2 + 6m$ and the number of inefficient vertices is at most $3m$, the number of $a$ angles in the tiling is at most $n/2 + 6m + 3(3m)$, and therefore 
$$
n/2 + 6m + 9m \geq n. 
$$
Solving this yields $30m \geq n$; if this is not the case there will not be enough $a$ angles in the tiling. But this contradicts Proposition \ref{generalQuadRatio}: that the ratio of efficient pentagons to non-convex quadrilaterals must be at least $48.14$. Therefore it must be the case that the perimeter per tile is greater than or equal to that of a Cairo/Prismatic pentagon.
\end{proof}

\begin{remark}
\emph{Note that the proof of Theorem \ref{oneeffn&mTorus} assumes at points that the non-convex quadrilaterals are spread out in the tiling, and at other points assumes they are densely packed. This is done in order to ensure the bounds on the number of inefficient vertices are correct. The proof may be improved if these assumptions are either avoided or treated in some way.}
\end{remark}

\begin{remark}
\emph{ An earlier draft of this paper showed the above results only for dihedral tilings with a single efficient pentagon and a single non-convex quadrialteral. We tried several other initial approaches to generalize the above proposition before finding the current one. Initially we attempted to find ratios of quadrilaterals which shared an edge with the efficient pentagons and those that didn't, as the first type have high perimeter. Additionally we attempted to find the perimeter-minimizing pentagon with five angles that tile but were unable to do so.}
\end{remark}

\section{Dihedral Tiling of Efficient and Type 2 Non-Convex Pentagons}

We adapt the technique used in Section 5 to the dihedral case with one efficient pentagon and one Type 2 non-convex pentagon. 

First recall from Section 2 that we have the following lemma:

\begin{replemma}{ncvx-pent-max-edge}
For a unit-area Type 2 non-convex pentagon, some edge must be greater than or equal to $\sqrt{2}$.
\end{replemma}

\begin{proposition}
\label{nonconvex-type2-not-surrounded}
In a unit-area dihedral tiling by Type 2 non-convex pentagons and efficient pentagons, a non-convex pentagon cannot be surrounded by efficient pentagons.
\end{proposition}

\begin{proof}
By Lemma \ref{ncvx-pent-max-edge}, some edge of the non-convex pentagon exceeds $\sqrt{2} > 1.41$. But by Lemma \ref{quad-pentagon-min-edge}, the longest edge in an efficient pentagon is less than $1.081$.
\end{proof}

\begin{proposition}
\label{type2NonconvexPerimeterMin}
In a dihedral tiling by an efficient pentagon and a Type 2 non-convex pentagon, the non-convex pentagon has perimeter at least 4.93594.
\end{proposition}

\begin{proof}
By Lemma \ref{quad-pentagon-min-edge}, the largest an edge of a unit-area efficient pentagon can be is 1.081. Therefore a non-convex pentagon in a dihedral tiling with an efficient pentagon must have an edge measuring at least 1.081. By Proposition \ref{EdgeMinTriangle}, the perimeter-minimizing unit-area triangle given one edge-length, $e$, is the isosceles triangle with base $e$. Since the non-convex pentagon will have perimeter greater than or equal to this triangle, we simply solve for the perimeter of such a triangle with edge-length 1.081, which is at least 4.93594.
\end{proof}

\begin{proposition}
\label{type2dihedralSomeDeg4Verts}
In a dihedral tiling of a torus by efficient pentagons and Type 2 non-convex pentagons, if the ratio of efficient pentagons to non-convex pentagons exceeds 24, then the tiling will have at least one degree-four efficient vertex. 
\end{proposition}

\begin{proof}
Consider a tiling of a flat torus by $n$ efficient pentagons and $m$ non-convex pentagons. The area of the torus is $n + m$, so by the Euler characteristic formula there are $3(n+m)/2$ vertices. By Proposition \ref{nonconvex-type2-not-surrounded}, the non-convex pentagons cannot be surrounded entirely by efficient pentagons. Therefore each non-convex pentagon shares an edge with at least one other non-convex pentagon, so there are at most eight inefficient vertices for every two non-convex pentagons. 

Two non-convex pentagons which share an edge have at most four small angles and four large angles between them. Assume there are two efficient pentagons at all four of the large angles. Then all the large angles are less than $198.16^\circ$ by Proposition \ref{minmaxangle}. So the four small angles average $71.84^\circ$, and since at least one angle must be greater than or equal to the average, at one of the small angles there are only three efficient pentagons. At the rest of the small angles there are four, so the total number of efficient pentagons around two non-convex pentagons is at most $4(2) + 3(4) + 3 - 8 = 15$ (subtracting eight because we assume the tiling is edge-to-edge). 

Suppose there were not two efficient pentagons at all the large angles. Then at least two of the large angles are greater than $198.16^\circ$ and there is only one efficient pentagon at these two large angles. There are two efficient pentagons at the other two large angles, and four at the four small angles for a total of $14$. We conclude that there are at most $15$ efficient pentagons around each pair of non-convex pentagons.

Let $k_3$ and $k_4$ be the number of degree three and degree four efficient vertices. Then counting each efficient vertex as one-fifth of a pentagon, we know
$$
(3/5)k_3 + (4/5)k_4 \geq n - (15/5)(m/2).
$$
We subtract $(15/5)(m/2)$ as this is the maximum number of vertices which can tile non-efficient vertices. Additionally
$$
3n/2 + 3m/2 \geq k_3 + k_4,
$$
as there cannot be more efficient vertices then the total number of vertices in the tilng. Thus $k_4 \geq n/2 - 12m.$ Therefore $k_4 \geq n/2 - 12m > 0$ when $n > 24m$.
\end{proof}

We adapt Proposition \ref{ratio1} to the case of dihedral tilings with efficient  and Type 2 non-convex pentagons.

\begin{proposition}
\label{type2-ratio}
In a unit-area dihedral tiling $T$ of a flat torus by an efficient pentagon and a Type 2 non-convex pentagon, with perimeter per tile less than or equal to half the perimeter of a Prismatic pentagon, the ratio of efficient pentagons to non-convex pentagons is greater than $34.77$.
\end{proposition}

\begin{proof}
We adapt a proof given by Chung et al. \cite[Prop. 2.11]{tile11} and Proposition \ref{ratio1}. By Proposition \ref{angles-convex-between}, $T$ cannot have an efficient pentagon with angles strictly between $\pi/2$ and $2\pi/3$. 
By Proposition \ref{type1-cvx-best}, the least-perimeter unit-area pentagon with angles not strictly between $\pi/2$ and $2\pi/3$ is defined as a Type 1 special convex pentagon. By Definition \ref{type1-cvx-best} this has perimeter, $P_0$, greater than $3.819$. The perimeter of a Cairo/Prismatic pentagon is $P_1 = 2\sqrt{2+\sqrt{3}} < 3.86$. The convex pentagons have perimeter at least $P_0$ by definition, and the non-convex pentagons have perimeter greater than $P_2 = 4.93594$ by Proposition \ref{type2NonconvexPerimeterMin}. 

Let $n$ and $m$ denote the number of efficient pentagons and non-convex pentagons. By hypothesis, 
$$ nP_0 + mP_2 < (n+m)P_1. $$
Therefore
$$ n/m > \frac{P_2-P_1}{P_1-P_0} \approx 24.117.$$

Proposition \ref{type2dihedralSomeDeg4Verts} implies that the tiling must contain at least one degree four vertex. But then by Proposition \ref{perimeterMinimizingDegreeFourPentagon}, the efficient pentagon cannot have perimeter exceeding $3.8328$. Substituting $3.8328$ in for $P_0$ yields $n/m > 34.77$. 
\end{proof}

\begin{proposition}
\label{type2-four-angles-don't-tile}
Consider a unit-area dihedral tiling of a flat torus by an efficient pentagon and a Type 2 non-convex pentagons. Assume that each efficient pentagon has at most four angles which tile with efficient pentagon's angles. Then the tiling has more perimeter per tile than half the perimeter of a Prismatic pentagon.
\end{proposition}

\begin{proof}
Because the efficient pentagon has at most four angles which tile, it cannot be surrounded entirely by efficient pentagons. Therefore the ratio of efficient pentagons to non-convex pentagons is at most the maximum number of efficient pentagons which can surround a non-convex pentagon.

At the three angles of the non-convex pentagon which are less than $\pi$, there are at most four efficient pentagons. If there were five or more, then the angles of the efficient pentagon would be too small, in violation of Proposition \ref{minmaxangle}. By the same logic, there are at most two efficient pentagons at the two angles in the non-convex pentagon which is greater than $\pi$. Since the tiling is edge-to-edge, five of the efficient pentagons surrounding the non-convex pentagon appear at two vertices, and we must avoid double counting these. So the total number of efficient pentagons surrounding the non-convex pentagon is $4(3) + 2(2) - 5 = 11$. Thus the ratio of efficient pentagons to non-convex pentagons is at most eleven to one. Therefore by Proposition \ref{type2-ratio} the tiling has more perimeter per tile than half that of a Prismatic pentagon.
\end{proof}

\begin{theorem}
\label{dihedraltype2n&mTorus}
A unit-area dihedral tiling of a flat torus by an efficient pentagon and a Type 2 non-convex pentagon cannot have perimeter per tile less than half the perimeter per tile of a Cairo/Prismatic tiling.
\end{theorem}

\begin{proof}
Assume there exists a unit-area tiling of a flat torus by $n$ efficient pentagons and $m$ Type 2 non-convex pentagons with perimeter per tile less than a Cairo/Prismatic tiling's. Note that $m \not= 0$; otherwise the tiling would contradict Theorem \ref{chungthm}. The area of the torus is $n + m$, so by the Euler characteristic formula there are $3(n + m)/2$ vertices. By Proposition \ref{nonconvex-type2-not-surrounded}, the non-convex pentagon cannot be surrounded entirely by efficient pentagons because at least one edge in the non-convex pentagon is too long. Therefore each non-convex pentagon shares an edge with at least one other non-convex pentagon, so there are at most eight inefficient vertices for each pair of non-convex pentagons. Therefore there are at most $4m$ inefficient vertices, so the number of efficient vertices is at least 
$$
\frac{3n}{2} + \frac{3m}{2} - 4m = \frac{3n}{2} - \frac{5m}{2}
$$
Let $k_3$ and $k_4$ be the number of degree three and degree four efficient vertices. By Corollary \ref{convex-vertices-degnot-2}, these are the only two types of efficient vertices which can appear in the tiling. Thus
$$
k_3 + k_4 \geq \frac{3n}{2} - \frac{5m}{2}.
$$
Additionally, since there are $n$ efficient pentagons, considering each vertex as a fifth of a pentagon we conclude that
$$
\frac{3}{5} k_3 + \frac{4}{5} k_4 \leq n.
$$
Therefore that $k_3 \geq n - 10m$ and $k_4 \leq n/2 + (15/2)m$.

Now by Proposition \ref{type2-ratio} it must be the case that $n > 34.77m$, as the tiling has perimeter per tile less than a Cairo pentagon's. By Proposition \ref{type2dihedralSomeDeg4Verts}, there exists at least one efficient vertex of degree four in the tiling. So there must be at least one angle, $s$, in the efficient pentagon which measures less than or equal to $90^\circ$. 

Suppose $s$ is the only angle that tiles an efficient vertex of degree four. Then $s = 90^\circ$. Therefore at the large angles of the non-convex pentagon there can be at most one efficient angle, and at small angles there can be at most three. As the non-convex pentagons are paired because of their edge-lengths, there are at most $4(1) + 4(3) - 8 = 8$ efficient pentagons (subtracting eight because we assume the tiling is edge-to-edge) per pair of non-convex pentagons. Therefore counting each efficient pentagon as one-fifth of a vertex, we know
$$
(3/5)k_3 + (4/5)k_4 \geq n - (8/5)(m/2).
$$
Additionally
$$
3n/2 + 3m/2 \geq k_3 + k_4,
$$
as there cannot be more efficient vertices than total vertices in the tiling. We can use these two inequalities to determine that $k_4 \geq n/2 - (17/2)m.$ 

We will show this implies there are too many $s$ angles in the tiling. As $s$ is the only angle that can tile a degree-four efficient vertex, there are at least $4k_4 \geq 2n - 34m$ $s$ angles. But as there cannot be more than $n$ $s$ angles, we have that $n \geq 4k_4$, and so $n \geq 2n-34m$. Thus $34m \geq n$. But since $n > 34.77m$ by Proposition \ref{type2-ratio} this inequality is false, i.e. there are more than $n$ $s$ angles. As this is a contradiction, there must be at least two angles which tile a degree four vertex.

Next we show there cannot be more than one angle other than $s$ which tiles a degree four efficient vertex. Let $a$, $b$, and $c$ be three angles in the efficient pentagon different than $s$. If $a + b + c + s = 360^\circ$ then the four angles average $90^\circ$ and the pentagon is not efficient by Proposition \ref{efficient-pentagon-one-small-angle}. If $2s + a + b = 360^\circ$, then by Lemma \ref{cot-perimeter-lemma} the perimeter is minimized when two angles measure $(360 - 2s)/2$, two measure $(180 + s)/2$, and one measures $s$. By definition $s \leq 90^\circ$ and by Proposition \ref{minmaxangle} $s > 80.92^\circ$; it follows from Lemma \ref{cot-perimeter-lemma} that the pentagon is not efficient. The case when $2a + s + b = 360^\circ$ is similar: just replace $s$ with $a$ and let $a$ range from $80.92^\circ$ to $142.29^\circ$ by Proposition \ref{minmaxangle}. As before, the pentagon is not efficient in this case. Therefore only one angle, say $a$, can tile a degree four vertex with $s$.

If there are two $s$ and two $a$ angles, then they average exactly $\pi/2$. Since the pentagon is efficient, by Lemma \ref{cot-perimeter-lemma} and Proposition \ref{efficient-pentagon-one-small-angle} it must be a Cairo or Prismatic pentagon, and the proposition follows as $m \not= 0$. 

Therefore there are either three $s$ or three $a$ angles at a degree four efficient vertex. First suppose there are three $a$ angles and one $s$ angle. Then $3a + s = 360$. By Proposition \ref{minmaxangle}, $s$ must be greater than $80.92^\circ$, and by hypothesis less than or equal to $90^\circ$. It follows that the average of $a$ and $s$ is between $86.973^\circ$ and $90^\circ$. Then by Proposition \ref{efficient-pentagon-one-small-angle}, the efficient pentagon must be Cairo or Prismatic and the proposition follows as $m \not= 0$. Therefore it must be the case that $3s + a = 360$. 

Because the tiling is dihedral, there are exactly $n$ $s$ angles. At all $k_4$ vertices there will be three $s$ angles. By Proposition \ref{minmaxangle}, there are at most $2(4) + 4(4) - 8 = 16$ efficient pentagons surrounding each pair of non-convex pentagons. Therefore counting each efficient vertex as one-fifth of a pentagon, we know
$$
(3/5)k_3 + (4/5)k_4 \geq n - (16/5)(m/2).
$$
Additionally
$$
3n/2 + 3m/2 \geq k_3 + k_4
$$
which we can use to determine that $k_4 \geq n/2 - (25/2)m.$

Then since there are three $s$ angles at each $k_4$ vertex, there are at least $3k_4 \geq 3n/2 - (75/2)m$ $s$ angles. But as there cannot be more than $n$ $s$ angles, we have that $n \geq 3k_4$, and so $n \geq 3n/2 - (75/2)m$ which means $n \leq 75m$.

In order to have a ratio of efficient pentagons to non-convex pentagons less than or equal to 75, the efficient pentagon must have perimeter less than $3.8495$. Here it is convenient to note that since the tiling has perimeter per tile less than Cairo/Prismatic, by Proposition \ref{type2-four-angles-don't-tile} the efficient pentagon must have five angles which tile. Then the efficient pentagon satisfies the hypothesis of Proposition \ref{quadCan'tTileDeg3}, so it will not be the case that $a$ tiles an efficient vertex of degree three. Therefore $a$ appears only at degree four efficient vertices and inefficient vertices.

There will be at most one $a$ angle at each degree four vertex. Because at least two-fifths of the inefficient angles have a large angle at them, there is at most one $a$ angle at at least two-fifths of the non-efficient vertices, and at most three $a$ angles at at most three-fifths of the non-efficient vertices, as $a$ is greater than $90^\circ$. Since $k_4 \leq n/2 + (15/2)m$ and the number of inefficient vertices is at most $4m$, the number of $a$ angles in the tiling is at most $n/2 + (15/2)m + (2/5)(4m) + (3)(3/5)(4m)$, and therefore 
$$
n/2 + (163/10)m \geq n. 
$$
Solving this yields $32.6m \geq n$; if this is not the case there will not be enough $a$ angles in the tiling. But this contradicts Proposition \ref{type2-ratio}: that the ratio of efficient pentagons to Type 1 non-convex pentagons must be at least $34.77$. Therefore it must be the case that the perimeter per tile is greater than or equal to that of a Cairo/Prismatic pentagon.
\end{proof}

\begin{remark}
\emph{We have to slightly improve two important values in this section to get the proof to apply to tilings by an efficient pentagon and any number of non-convex Type 2 pentagons to work. We need to get the ratio of efficient to non-convex Type 2 pentagons which implies there exists at least one degree four vertex down to 22.57 from 24, and we need to get the upper-bound for the ratio of efficient to non-convex Type 2 pentagons in a tiling which is better than Cairo/Prismatic tilings down to 31.26.}
\end{remark}

\begin{remark}
\emph{Considering trihedral tilings is a logical in generalizing the results from this section. While we have solved the case of one efficient pentagon two non-convex Type 2 pentagons in Theorem \ref{dihedraltype2n&mTorus}, there is still the case of two different efficient pentagons and a Type 2 non-convex pentagon, and the case of an efficient pentagon, a non-efficient convex pentagon, and a Type 2 non-convex pentagon, both of which remain open.}
\end{remark}


\section{Extension of Results to the Plane}

This section contains early attempts and work to extend the above results, many of which hold only for large flat tori, to the plane. We first generalize the concept of the perimeter, as considered on a flat torus, to the plane.

\begin{definition}
\label{def:perimeterRatio}
\emph{The} perimeter ratio, $\rho$, \emph{of a planar tiling is the limit supremum as $R$ goes to infinity of the perimeter inside a disk of radius $R$ about the origin , divided by $\pi R^2$. We intend to prove that this does not depend on the choice of origin.}
\end{definition}

The following result allows us to generalize results from finite cases on large tori to infinite cases in the plane.

\begin{lemma}[Truncation Lemma] \cite[15.3]{morggeo}
\label{truncation-lemma}
Consider a tiling of the plane. Let $P(R)$ and $A(R)$ be the perimeter and area of the tiling in a disk of radius $R$ centered at the origin. Additionally let $P_0(R)$ and $A_0(R)$ be the perimeter and area of tiles completely contained within the disk. Then given $\epsilon > 0$, $R_0 > 0$, there exists some $R \geq R_0$ such that the number $N(R)$ of points where the circle $S(0, R)$ intersects edges of tiles is less than $\epsilon A(R)$. In particular, the number of tiles which intersect the disk but are not contained in the disk is at most $\epsilon A(R)$. Furthermore 
$$
\liminf_{R \to \infty} \frac{P_0}{A_0} \leq \rho.
$$
\end{lemma}

The following proposition generalizes Proposition \ref{ratio1} to the plane:

\begin{proposition}
\label{plane-ext-ratio1}
Let $T$ be a tiling of the plane by unit-area pentagons, with perimeter ratio less than or equal to half the perimeter of a Prismatic pentagon. Then the fractions $C_1$, $N_1$ and $N_2$ of the pentagons completely inside a disk about the origin of radius $R$ which are efficient, non-convex Type 1, and non-convex Type 2, respectively, satisfy:   
$$
\liminf_{R \to \infty} (C_1 - 2.6N_1 + 13.4N_2) \geq 0.
$$ 
\end{proposition}

\begin{proof}
The perimeters of a regular pentagon, a Cairo/Prismatic pentagon, the unit square, and the unit-area equilateral triangle are $P_1 = 2\sqrt{5} \sqrt[4]{5 - 2\sqrt{5}}$, $P_2 = 2\sqrt{2 + \sqrt{3}}$, $P_3 = 4$ and $P_4 = 3 \sqrt{4/\sqrt{3}}$. 

Let $\epsilon > 0$. By the definition of $\rho$, there exists an $R_0 > 0$ such that for $R_1 > R_0$ the perimeter to area ratio of the disk of radius $R_1$ is less than $\rho + \epsilon$. By the Truncation Lemma \ref{truncation-lemma} there exists an $R > R_0$ such that the perimeter $P_0(R)$ and area $A_0(R)$ of the tiles completely within the disk of radius $R$ satisfy
$$
\frac{P_0}{A_0} < \rho + \epsilon.
$$
Then since the efficient, non-efficient convex, Type 1 non-convex, and Type 2 non-convex pentagons have perimeter at least $P_1, P_2, P_3$ and $P_4$ respectively,  
$$
C_1 P_1 + C_2 P_2 + N_1 P_3 + N_2 P_4 \leq 2\frac{P_0}{A_0}.
$$
Therefore $C_1 P_1 + C_2 P_2 + N_1 P_3 + N_2 P_4 \leq 2(\rho + \epsilon)$, which by hypothesis implies
$$
C_1 P_1 + C_2 P_2 + N_1 P_3 + N_2 P_4 < P_2 + 2\epsilon. 
$$
Then,
$$
C_1 P_1 + C_2 P_2 + N_1 P_3 + N_2 P_4 < P_2(C_1 + C_2 + N_1 + N_2) + 2\epsilon, 
$$
and
$$
C_1 > N_1 \frac{P_3 - P_2}{P_2 - P_1} + N_2 \frac{P_4 - P_2}{P_2 - P_1} - \frac{2\epsilon}{P_2 - P_1}.
$$
By the definition of the $P_i$,
$$
C_1 - 2.6N_1 - 13.4N_2 > -40\epsilon.
$$
By the definition of the limit inferior the proposition holds.
\end{proof}

\begin{proposition}
\label{plane-ext-stronger-ratio1}
Let $T$ be a tiling of the plane with only efficient and non-convex pentagons, with perimeter ratio less than or equal to half the perimeter of a Prismatic pentagon. Then the fractions $C_1$, $N_1$, and $N_2$ of the pentagons completely inside a disk about the origin of radius R which are efficient, non-convex Type 1, and non-convex Type 2, respectively, satisfy:
$$
C_1 > 2.6N_1 + 13.4N_2.
$$ 
\end{proposition}

\begin{proof}
The perimeters of a regular pentagon, a Cairo/Prismatic pentagon, the unit square, and the unit-area equilateral triangle are $P_1 = 2\sqrt{5} \sqrt[4]{5 - 2\sqrt{5}}$, $P_2 = 2\sqrt{2 + \sqrt{3}}$, $P_3 = 4$ and $P_4 = 3 \sqrt{4/\sqrt{3}}$. 

Let $\epsilon > 0$. By the definition of $\rho$, there exists an $R_0 > 0$ such that for $R_1 > R_0$ the perimeter to area ratio of the disk of radius $R_1$ is less than $\rho + \epsilon$. By the Truncation Lemma \ref{truncation-lemma} there exists an $R > R_0$ such that the perimeter $P_0(R)$ and area $A_0(R)$ of the tiles completely within the disk of radius $R$ satisfy
$$
\frac{P_0}{A_0} < \rho + \epsilon.
$$
Then since the efficient, Type 1 non-convex, and Type 2 non-convex pentagons have perimeter at least $P_1, P_3$ and $P_4$ respectively,  
$$
C_1 P_1 + N_1 P_3 + N_2 P_4 \leq 2\frac{P_0}{A_0}.
$$
Therefore $C_1 P_1 + N_1 P_3 + N_2 P_4 \leq 2(\rho + \epsilon)$, which by hypothesis implies
$$
C_1 P_1 + N_1 P_3 + N_2 P_4 < P_2 + 2\epsilon. 
$$
Then,
$$
C_1 P_1 + N_1 P_3 + N_2 P_4 < P_2(C_1 + N_1 + N_2) + 2\epsilon, 
$$
and
$$
C_1 > N_1 \frac{P_3 - P_2}{P_2 - P_1} + N_2 \frac{P_4 - P_2}{P_2 - P_1} - \frac{2\epsilon}{P_2 - P_1}.
$$
Then since $C_1 + N_1 + N_2 = 1$,
$$
C_1 > 2.63 N_1 + 13.43 N_2 - 38.63 \epsilon (C_1 + N_1 + N_2),
$$
and therefore
$$
C_1 > \frac{2.63 - 38.63 \epsilon}{1 + 38.63 \epsilon} N_1 + \frac{13.43 - 38.63 \epsilon}{1 + 38.63 \epsilon} N_2 
$$
Choosing any value of $\epsilon < 5.39 \times 10^{-5}$ implies
$$
C_1 > 2.6N_1 + 13.4N_2.
$$ 
\end{proof}

We extend Proposition \ref{angles-convex-between} to the plane :

\begin{proposition}
\label{plane-ext-angles-convex-between}
A unit-area tiling of the plane by non-convex pentagons and pentagons with angles strictly between $\pi/2$ and $2\pi/3$ has a perimeter ratio more than half the perimeter of a Prismatic pentagon.
\end{proposition}

\begin{proof}
By the Truncation Lemma \ref{truncation-lemma}, we can find some $R$ large enough so that for tiles within a/ disk of radius $R$, the perimeter per tile is less than half the perimeter of a Prismatic pentagon. Note that the pentagons in $D_R$ will be a finite collection.

Assume, on the contrary, that there existed some large $R$ such that a tiling of $D_R$ by non-convex pentagons and efficient pentagons with angles strictly between $\pi/2$ and $2\pi/3$ had a perimeter ratio less than half the Prismatic pentagon's. By Proposition \ref{plane-ext-ratio1}, the ratio of the number of efficient pentagons to the number of non-convex pentagons must be greater than 2.6. Since all the angles of the efficient pentagons are strictly between $\pi/2$ and $2\pi/3$, there is at least one non-convex pentagon at each vertex. By definition, a non-convex pentagon has at least one angle greater than $\pi$. Thus at least $1/5$ of the vertices must contain an angle greater than $\pi$. At such vertices there is at most one efficient pentagon. At the remaining vertices, there are at most three efficient pentagons, because their angles are greater than $\pi/2$. At vertices near the boundary, where truncation occurs, these bounds may not hold, as some pentagons may be removed. However by the Truncation Lemma \ref{truncation-lemma} for any $\epsilon$ we can find some $R$ such that the number of pentagons being removed is $\epsilon A(R)$. Therefore for large $R$ the number of such pentagons is insignificant. Thus the ratio of efficient pentagons to non-convex pentagons is at most $3(4/5)+1(1/5) = 2.6$.  This is a contradiction of Proposition \ref{plane-ext-ratio1}, which says the ratio of efficient to non-convex pentagons must be strictly greater than 2.6.
\end{proof}

We extend Proposition \ref{four-angles-don't-tile} to the plane:

\begin{proposition}
\label{plane-four-angles-don't-tile}
Consider a unit-area tiling of a the plane by efficient pentagons and Type 2 non-convex pentagons. Assume that each efficient pentagon has at most four angles which tile with efficient pentagons. Then the tiling has a perimeter ratio greater than half the perimeter of a Prismatic pentagon.
\end{proposition}

\begin{proof}
At the three angles of a non-convex pentagon which are less than $\pi$, there are at most four efficient pentagons. If there were five or more, then the angles of the efficient pentagon would be too small, in violation of Proposition \ref{minmaxangle}. By the same logic, there are at most two efficient pentagons at the two angles in the non-convex pentagon which is greater than $\pi$. Since the tiling is edge-to-edge, five of the efficient pentagons surrounding a non-convex pentagon appear at two vertices, and we must avoid double counting these. So the maximum number of efficient pentagons surrounding a non-convex pentagon in the tiling is $4(3) + 2(2) - 5 = 11$. 

Now given $\epsilon > 0$, by Lemma \ref{truncation-lemma} there is a large disk of radius $R$ such that the number of pentagons which intersect the disk but are not contained within it is at most $\epsilon A(R)$. Let $C_1$ and $N_2$ be the fraction of efficient and Type 2 non-convex pentagons with at least one vertex in the disk. Note that by hypothesis each efficient pentagon within shares a vertex with at least one non-convex pentagon, and by the conclusion of the first paragraph there are at most 11 efficient pentagons which surround each non-convex pentagon, so we might expect that $C_1 < 11N_2.$ However near the boundary of the disk we may have efficient pentagons within the disk which share a vertex with a non-convex pentagon which is not contained within the disk. There are at most $\epsilon A(r)$ non-convex pentagons of this type, so in fact we have 
$$
C_1 < 11N_2 - \epsilon N_2.
$$

Therefore $C_1 < 11N_2/(1 - \epsilon)$.
But if $\epsilon < 1 - 11/13.4$ and the perimeter ratio of the tilings is less than or equal to half the perimeter of a Prismatic pentagon this contradicts Proposition \ref{plane-ext-stronger-ratio1} as $C_1 \not> 13.4N_2$. Therefore the tiling has a perimeter ratio greater than half that of a Prismatic pentagon.
\end{proof}


\section{Projects for Future Study}

This section marks future projects which might be used to remove the convexity assumption in Theorem \ref{chungthm} and prove the main conjecture:

\begin{conjt}
\label{mainConjt}
Perimeter-minimizing planar tilings by unit-area polygons with at most five sides are given by Cairo and Prismatic tiles.
\end{conjt}

\begin{proj}
\emph{Extend the results from Section 5 and 6 on flat tori to similar results on the plane, considering methods from Morgan \cite{morggeo} and Chung et al. \cite{pen11}. See Section 7 for partial progress.}
\end{proj}

The following may be a more accessible version of Conjecture \ref{mainConjt}:

\begin{projs}
\emph{(1) Prove \ref{mainConjt} on flat tori of area $n$ for small $n$, known for $n = 4$ (\cite[Prop. 3.9]{tile11}) and some tori of area 2 (\cite[Prop. 3.3]{tile11}). 
\newline
(2) Prove \ref{mainConjt} for dihedral tilings (the case with Type 1 non-convex pentagons remains open). 
\newline
(3) Prove \ref{mainConjt} for tilings of the plane by convex pentagons and Type 2 pentagons.}
\end{projs}

The following may be helpful in proving some of the above or \ref{mainConjt}:

\begin{projs}
\emph{(1) Prove that if a tiling by convex and non-convex pentagons has less perimeter than Cairo tilings, then the convex pentagons must have at least one angle of $90^\circ$ or $120^\circ$. 
\newline
(2) Without using Theorem \ref{chungthm}, show a tiling of a large flat torus by unit-area convex pentagons and unit squares must have more perimeter per tile then one half the Prismatic pentagon's perimeter. 
\newline
(3) Without using Theorem \ref{chungthm}, show a tiling of a large flat torus by unit-area convex pentagons and unit-area equilateral triangles must have more perimeter per tile then one half the Prismatic pentagon's perimeter.}
\end{projs}

\appendix

\section{Perimeter-minimizing Equilateral Pentagonal Tilings}

Theorem \ref{equilateralPentagonMin} provides the perimeter-minimizing monohedral tiling of the plane by equilateral pentagons. We begin with a description of equilateral pentagonal convex tiles given by Hirschhorn and Hunt.
 
\begin{theorem}
\label{hirsch-hunt-thm}
\cite{hirsch&hunt} An equilateral convex pentagon tiles the plane if and only if it has two angles adding to $\pi$, as in Figure \ref{fig:EquilCvxPentCases}, or it is the unique equilateral convex pentagon $X$ of Figure \ref{fig:spCvxPentX} with angles $A, B, C, D, E$ satisfying $A+2B = 2\pi,$ $C+2E= 2\pi,$ $A+C+2D = 2\pi$. ($A \approx 70.88^\circ, B \approx 144.56^\circ, C \approx 89.26^\circ, D \approx 99.93^\circ, E \approx 135.37^\circ.$)
\end{theorem}

\begin{figure}
\centering
\includegraphics[scale=.4]{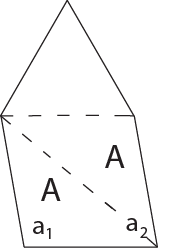}
\includegraphics[scale=.5]{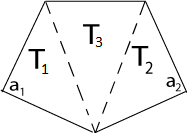}
\caption{There are two cases when an equilateral pentagon has two angles, $a_1$ and $a_2$, summing to $\pi$.}
\label{fig:EquilCvxPentCases}
\end{figure}

We now turn our attention to perimeter-minimizing equilateral pentagons which tile the plane monohedrally, and provide the perimeter-minimizing pentagon of this type. If we remove the monohedral assumption or allow for the addition of non-convex tiles we might be able to do better -- this question remains open . 

\begin{theorem}
\label{equilateralPentagonMin}
The perimeter-minimizing unit-area equilateral convex pentagon $P$ which tiles the plane monohedrally is circumscribed about a sphere and has two non-adjacent $\pi/2$ angles, two adjacent angles measuring $\pi/4 + \arccos(1/2\sqrt{2}) \approx 1.995$ and one angle measuring $3\pi/4 - 2\arccos(1/2\sqrt{2}) \approx 2.294.$ $P$ has perimeter approximately 3.879.
\end{theorem}

\begin{proof}
We first show that such a perimeter-minimizing tiling exists. Consider a sequence of such tilings $T_n$ with perimeter per tile approaching the infimum. By Definition \ref{def:CairoPrismaticEfficient}, we may assume that the pentagons are convex, as $P$ has perimeter less than four. By standard compactness, the desired limit exists.

By Theorem \ref{hirsch-hunt-thm} an equilateral convex pentagon which tiles the plane must be one of three types: a pentagon with two adjacent angles summing $\pi$, a pentagon with two non-adjacent angles summing to $\pi$, or the special pentagon $X$.

First consider a unit-area equilateral convex pentagon where the two angles, $a_1$ and $a_2$, which sum to $\pi$ are adjacent, and the side length is denoted $x_1$. 
Note that pentagons of this type are formed by a parallelogram that can be divided into two congruent triangles, each denoted $A$, and an equilateral triangle, as in Figure \ref{fig:EquilCvxPentCases}.  
$A$ is isosceles, with two sides of length $x_1$; therefore $A$ has a base of length $2x_1\sin (a_1/2)$ and a height of $x_1\cos (a_1/2)$. 
Because the pentagon has unit area, we can express $x_1$ in terms of $a_1$ by solving the equation 
$$
1 = x_1^2 \left( 2\sin(a_1/2)\cos(a_1/2) + \sqrt{3}/4 \right),
$$
to get
$$
x_1^2 = \frac{1}{\sin a_1+\sqrt{3}/4}.
$$
Note that $x_1$ is minimized when $a_1 = \pi/2$.

\begin{figure}
\centering
\includegraphics[scale=.5]{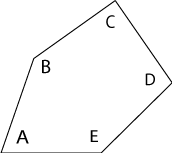}
\caption{The special equilateral pentagonal tile $X$ of Theorem \ref{hirsch-hunt-thm}}
\label{fig:spCvxPentX}
\end{figure}

The perimeter is clearly just $5x_1$, so the perimeter is minimized when $a_1 = \pi/2$. 
Solving the above equation to determine this value yields $x_1 \approx .8353$, so the perimeter of $P_1$ is approximately 4.177.
\newline

Second consider an equilateral convex pentagon where the two angles, $a_1$ and $a_2$, are non-adjacent. Let $x_2$ be the side-length of the perimeter-minimizing pentagon, $P_2$, of this type. 
This pentagon can be broken into three triangles, as in Figure \ref{fig:EquilCvxPentCases}: two of which have two sides of length $x_2$, denoted $T_1$ and $T_2$, and one of which has one side of length $x_2$, denoted $T_3$. 
We first show that $a_1 = a_2 = \pi/2$ maximizes the area of $T_3$.

Let $A_i$ be the side of $T_i$ opposite angle $a_i$. 
Since the $T_i$ are isosceles triangles, we can write $\sin (a_1/2) = A_1/2x_2$ and $\sin(a_2/2) = A_2/2x_2$. Substituting $a_2 = \pi-a_1$ we get that 
$$
\frac{A_2}{2x_2} = \sin{\frac{\pi-a_1}{2}} = \cos \frac{a_1}{2} = \frac{h}{x_2},
$$
where $h$ is the height of $T_1$. The Pythagorean Theorem implies $h = \sqrt{x_2^2 - (A_1^2/4)}$; therefore
$$
A_2 = 2\sqrt{x^2 - (A_1^2/4)}.
$$ 
The $T_3$ triangle has sides $A_1, A_2, x_2$ and by Heron's formula the area is
$$
|T_3| = \sqrt{p(p-A_1)(p-A_2)(p-x_2)},
$$
where $p$ is half the perimeter of $T_3$. Since $x_2$ is given and we can substitute for $A_2$, we use Mathematica to take the derivative of $|T_3|$ with respect to $A_1$ and get that the area is maximized when $x_2 = A_1/\sqrt{2}$. This is the case precisely when $a_1 = a_2 = \pi/2$, as desired.

We now show that $P_2$ must have $a_1 = a_2 = \pi/2$. Assume $a_1 \not= \pi/2$ in $P_2$. Setting $a_1 = a_2 = \pi/2$ increases the area of $T_1$ and $T_2$, and we can adjust the angle between $T_1$ and $T_2$ to ensure that $T_3$ still has a side of length $x_2$ (i.e. that the pentagon is still equilateral). Since this maximizes the area of $T_3$, the overall area of $P_2$ increases while the perimeter stays the same. Then scaling back down to unit-area by just shrinking the pentagon decreases the perimeter, a contradiction that $P_2$ is perimeter-minimizing. So $a_1 = a_2 = \pi/2$ in this case as well.

The area of $P_2$ is given by 
$$
1 = 2T_1 + T_2 = x_2^2 + \frac{\sqrt{7}x_2^2}{4},
$$ 
which implies that $x_2 \approx .7758$ and that the perimeter is about $3.879$.
\newline

For the final case, the special pentagon $X$ pictured in Figure \ref{fig:spCvxPentX}, we note that by Lemma \ref{cot-perimeter-lemma} the perimeter of $P_3$ is greater than 3.994.
Comparing the three cases, we find that $P_2$ has the lowest perimeter.

By definition $P_2$ had two non-adjacent $\pi/2$ angles. $T_1$ and $T_2$ each have two angles of $\pi/4$. Therefore $T_3$ has two angles which measure $\arccos(1/2\sqrt{2})$ and one which measures $\pi - 2\arccos(1/2\sqrt{2})$. By construction $P_2$ has two adjacent angles measuring $\pi/4 + \arccos(1/2\sqrt{2})$ and one angle measuring $3\pi/4 - 2\arccos(1/2\sqrt{2})$, as desired.
\end{proof}


\bibliographystyle{abbrv}
\bibliography{main}

\bigskip
Whan Ghang
\newline 
Department of Mathematics
\newline 
Massachusetts Institute of Technology
\newline
ghangh@MIT.edu
\newline
\newline
Zane Martin
\newline 
Department of Mathematics and Statistics 
\newline 
Williams College 
\newline
zkm1@williams.edu
\newline
\newline
Steven Waruhiu
\newline 
Department of Mathematics
\newline 
University of Chicago
\newline
waruhius@uchicago.edu

\end{document}